\documentclass[a4paper,11pt]{article}
\usepackage[top=1.25in, bottom=1.25in, left=3.0cm, right=3.0cm,footskip=0.25in]{geometry}

\usepackage{times}

 \RequirePackage{amsmath}
\usepackage{amssymb}
\usepackage{amsthm}
\usepackage{mathtools}
\newtheorem{theorem}{Theorem}[section]
\newtheorem{definition}[theorem]{Definition}
\newtheorem{proposition}[theorem]{Proposition}
\newtheorem{lemma}[theorem]{Lemma}

\newtheorem{corollary}[theorem]{Corollary}

\title{Some Hoeffding- and Bernstein-type Concentration Inequalities}

\author{{\bf Andreas Maurer} \\ Istituto Italiano di Tecnologia, 16163 Genoa, Italy \and {\bf Massimiliano Pontil}   \\ Istituto Italiano di Tecnologia, 16163 Genoa, Italy}
 
\begin{document}

\maketitle

\begin{abstract}%
	We prove concentration inequalities for functions of independent random variables under sub-Gaussian and sub-exponential conditions. The inequalities are applied to principal subspace analysis, generalization bounds with Rademacher complexities and Lipschitz functions on unbounded metric spaces.%
\end{abstract}

\bigskip
\section{Introduction}
The popular bounded difference inequality \cite{McDiarmid98} has become a
standard tool in the analysis of algorithms. It bounds the deviation
probability of a function of independent random variables from its mean in
terms of the sum of conditional ranges, and may not be applied when these
ranges are infinite. This hampers the utility of the inequality in certain
situations. It may happen that the conditional ranges are infinite, but the
conditional versions, the random variables obtained by fixing all but one of
the arguments of the function, have light tails with exponential decay. In
this case we might still expect exponential concentration, but the bounded
difference inequality is of no help.

Vershyinin's book \cite{vershynin2018high} gives general Hoeffding and
Bernstein-type inequalities for sums of independent sub-Gaussian or
sub-exponential random variables. In situations, where the bounded
difference inequality is used, one would like to have analogous bounds for
general functions. In this work we use the entropy method{\ (\cite%
	{ledoux2001concentration}, \cite{boucheron2003concentration}, \cite%
	{Boucheron13})} to extend these inequalities from sums to general functions
of independent variables, for which the centered conditional versions are
sub-Gaussian or sub-exponential{,} respectively. These concentration
inequalities, Theorem \ref{Theorem subgaussian}, \ref{Theorem subexponential}
and \ref{Theorem Bernsteinoid}, are stated in Section \ref{Section results}
below. Theorems \ref{Theorem subexponential} and \ref{Theorem Bernsteinoid},
which apply to the {heavier} tailed sub-exponential distributions, are our
principal contributions. Theorem \ref{Theorem subgaussian} for the
sub-Gaussian case has less novelty, but it is included to complete the
picture, and because its proof provides a good demonstration of the entropy {%
	method.}

For the purpose of illustration we apply these results to some standard
problems in learning theory, vector valued concentration, the generalization
of PCA and the method of Rademacher {complexities}. Over the last twenty
years the latter method (\cite{Bartlett02}, \cite{Koltchinskii00}) has been
successfully used to prove generalization bounds in a variety of situations.
The Rademacher complexity itself does not necessitate boundedness, but, when
losses and data-distributions are unbounded, the use of the bounded
difference inequality can only be circumnavigated with considerable effort.
Using our bounds the extension is immediate. We also show how an inequality
of Kontorovich \cite{kontorovich2014concentration}, which describes
concentration on products of sub-Gaussian metric probability spaces and has
applications to algorithmic stability, can be extended to the
sub-exponential case.

\paragraph{Related work}

Several works contain results very similar to Theorem \ref{Theorem
	subgaussian}, which refers to the sub-Gaussian case. The closest to it is
Theorem 3 in \cite{Meir03}, which gives essentially the same learning bounds
for sub-Gaussian {distributions}. Theorem 1 in \cite%
{kontorovich2014concentration} is also somewhat similar, but specializes to
metric probability spaces. Somewhat akin is the work in \cite%
{kutin2002extensions}.

To address the sub-exponential case, we have not found results comparable to
Theorems \ref{Theorem subexponential} and \ref{Theorem Bernsteinoid} in the
literature.

There has been a lot of work to establish generalization in unbounded
situations (\cite{Meir03}, \cite{cortes2013relative}, \cite%
{kontorovich2014concentration}), or the astounding results in \cite%
{mendelson2014learning}, but we are unaware of an equally simple extension
of the method of Rademacher complexities to sub-exponential distributions,
as the one given below.

\section{Notation and Conventions}

We use upper-case letters for random variables and vectors of random
variables and lower case letters for scalars and vectors of scalars. In the
sequel $X=\left( X_{1},\dots ,X_{n}\right) $ is a vector of independent
random variables with values in a space $\mathcal{X}$, the vector $X^{\prime
}=\left( X_{1}^{\prime },\dots ,X_{n}^{\prime }\right) $ is iid to $X$ and $f
$ is a function $f:\mathcal{X}^{n}\rightarrow \mathbb{R}$. We are interested
in concentration of the random variable $f\left( X\right) $ about its
expectation, and require some special notation to describe the fluctuations
of $f$ in its $k$-th variable $X_{k}$, when the other variables $\left(
x_{i}:i\neq k\right) $ are given. 

\begin{definition}
	If $f:\mathcal{X}^{n}\rightarrow \mathbb{R}$, $x=\left(
	x_{1},...,x_{n}\right) \in \mathcal{X}^{n}$ and $X=\left(
	X_{1},...,X_{n}\right) $ is a random vector with independent components in $%
	\mathcal{X}^{n}$, then the $k$-th centered conditional version of $f$ is the
	random variable%
	\begin{equation*}
	f_{k}\left( X\right) \left( x\right) =f\left( x_{1},\dots
	,x_{k-1},X_{k},x_{k+1},\dots ,x_{n}\right) -\mathbb{E}\left[ f\left(
	x_{1},\dots ,x_{k-1},X_{k}^{\prime },x_{k+1},\dots ,x_{n}\right) \right] .
	\end{equation*}
\end{definition}

Then $f_{k}\left( X\right) $ is a random-variable-valued function $%
f_{k}\left( X\right) :x\in \mathcal{X}^{n}\mapsto f_{k}\left( X\right)
\left( x\right) $, which does not depend on the $k$-th coordinate of $x$. If 
$\left\Vert .\right\Vert _{a}$ is any given norm on random variables, then $%
\left\Vert f_{k}\left( X\right) \right\Vert _{a}\left( x\right) :=\left\Vert
f_{k}\left( X\right) \left( x\right) \right\Vert _{a}$ defines a
non-negative real-valued function $\left\Vert f_{k}\left( X\right)
\right\Vert _{a}$ on $\mathcal{X}^{n}$. Thus $\left\Vert f_{k}\left(
X\right) \right\Vert _{a}\left( X\right) $ is also a random variable, of
which $\left\Vert \left\Vert f_{k}\left( X\right) \right\Vert
_{a}\right\Vert _{\infty }$ is the essential supremum. If $X^{\prime }$ is
iid to $X$ then $\left\Vert f_{k}\left( X\right) \right\Vert _{a}$ is the
same function as $\left\Vert f_{k}\left( X^{\prime }\right) \right\Vert _{a}$
and $\left\Vert f_{k}\left( X\right) \right\Vert _{a}\left( X^{\prime
}\right) $ is iid to $\left\Vert f_{k}\left( X\right) \right\Vert _{a}\left(
X\right) $. Note that 
\begin{equation*}
f_{k}\left( X\right) \left( X\right) =f\left( X\right) -\mathbb{E}\left[
f\left( X\right) |X_{1},...,X_{k-1},X_{k+1},...X_{n}\right] .
\end{equation*}%
Also, if $f\left( x\right) =\sum_{i=1}^{n}x_{i}$, then $f_{k}\left( X\right)
\left( x\right) =X_{k}-\mathbb{E}\left[ X_{k}\right] $ is independent of $x$%
. 

It follows from Propositions 2.7.1 and 2.5.2 in \cite{vershynin2018high}
that we can equivalently redefine the usual sub-Gaussian and sub-exponential
norms $\left\Vert \cdot \right\Vert _{\psi _{2}}$ and $\left\Vert \cdot
\right\Vert _{\psi _{1}}$ for any real random variable $Z$ as%
\begin{equation}
\left\Vert Z\right\Vert _{\psi _{1}}=\sup_{p\geq 1}\frac{\left\Vert
	Z\right\Vert _{p}}{p}\text{ and }\left\Vert Z\right\Vert _{\psi
	_{2}}=\sup_{p\geq 1}\frac{\left\Vert Z\right\Vert _{p}}{\sqrt{p}},
\label{Modified Orlicz Norms}
\end{equation}%
where $\left\Vert \cdot \right\Vert _{p}$ are the usual $L_{p}$-norms. It
also follows from the {above mentioned} propositions that for every centered
sub-Gaussian random variable $Z$ we have, {for all $\beta \in \mathbb{R}$}, 
\begin{equation}
\mathbb{E}\left[ e^{\beta Z}\right] \leq e^{4e\beta ^{2}\left\Vert
	Z\right\Vert _{\psi _{2}}^{2}}.  \label{Subgaussian MGF bound}
\end{equation}

If $H$ is a Hilbert space, then the Hilbert space of Hilbert-Schmidt
operators $HS\left( H\right) $ is the set of bounded operators $T$ on $H$
satisfying $\left\Vert T\right\Vert _{HS}=\sqrt{\sum_{ij}\left\langle
	Te_{i},e_{j}\right\rangle _{H}^{2}}<\infty $ with inner product $%
\left\langle T,S\right\rangle _{HS}=\sum_{ij}\left\langle
Te_{i},e_{j}\right\rangle _{H}\left\langle Se_{i},e_{j}\right\rangle _{H}$,
where $\left( e_{i}\right) $ is an orthonormal basis. For $x\in H$ the
operator $Q_{x}\in HS\left( H\right) $ is defined by $Q_{x}y=\left\langle
y,x\right\rangle x$, and one verifies that $\left\Vert Q_{x}\right\Vert
_{HS}=\left\Vert x\right\Vert _{H}^{2}$.\bigskip 

\section{Results\label{Section results}}

Our first result assumes sub-Gaussian versions $f_{k}\left( X\right) $. It
is an unbounded analogue of the popular bounded difference inequality, which
is sometimes also called McDiarmid's inequality (\cite{Boucheron13}, \cite%
{McDiarmid98}).

\begin{theorem}
	\label{Theorem subgaussian}Let $f:\mathcal{X}^{n}\rightarrow \mathbb{R}$ and 
	$X=\left( X_{1},\dots ,X_{n}\right) $ be a vector of independent random
	variables with values in a space $\mathcal{X}$. Then for any $t>0$ we have%
	\begin{equation*}
	\Pr \left\{ f\left( X\right) -\mathbb{E}\left[ f\left( X^{\prime }\right) %
	\right] >t\right\} \leq \exp \left( \frac{-t^{2}}{32e\left\Vert
		\sum_{k}\left\Vert f_{k}\left( X\right) \right\Vert _{\psi
			_{2}}^{2}\right\Vert _{\infty }}\right) .
	\end{equation*}
\end{theorem}

If $f$ is a sum of sub-Gaussian variables this reduces to the general
Hoeffding inequality, Theorem 2.6.2 in \cite{vershynin2018high}. On the
other hand, if the $f_{k}\left( X\right) $ are a.s. bounded, $\left\Vert
f_{k}\left( X\right) \right\Vert _{\infty }\left( x\right) \leq r_{k}\left(
x\right) $, then also $\left\Vert f_{k}\left( X\right) \right\Vert _{\psi
	_{2}}\left( x\right) \leq r_{k}\left( x\right) $ and we recover the bounded
difference inequality (Theorem 6.5 in \cite{Boucheron13}) up to a constant
factor. A similar results to Theorem \ref{Theorem subgaussian} is given with
better constants in \cite{kontorovich2014concentration}, although in
specialized and slightly weaker forms, where the essential supremum is
inside the sum in the denominator of the exponent.\bigskip 

The next two results are our principal contributions and apply to functions
with sub-exponential conditional versions.

\begin{theorem}
	\label{Theorem subexponential}With $f$ and $X$ as in Theorem \ref{Theorem
		subgaussian} for any $t>0$%
	\begin{multline*}
	\Pr \left\{ f\left( X\right) -\mathbb{E}\left[ f\left( X^{\prime }\right) %
	\right] >t\right\}  \\
	\leq \exp \left( \frac{-t^{2}}{4e^{2}\left\Vert \sum_{k}\left\Vert
		f_{k}\left( X\right) \right\Vert _{\psi _{1}}^{2}\right\Vert _{\infty
		}+2e\max_{k}\left\Vert \left\Vert f_{k}\left( X\right) \right\Vert _{\psi
		_{1}}\right\Vert _{\infty }t}\right) .
\end{multline*}
\end{theorem}

The bound exhibits a sub-Gaussian tail governed by the variance-proxy $%
\left\Vert \sum_{k}\left\Vert f_{k}\left( X\right) \right\Vert _{\psi
	_{1}}^{2}\right\Vert _{\infty }$ for small deviations, and a sub-exponential
tail governed by the scale-proxy $\max_{k}\left\Vert \left\Vert f_{k}\left(
X\right) \right\Vert _{\psi _{1}}\right\Vert _{\infty }$ for large
deviations. If $f$ is a sum we recover the inequality in \cite%
{vershynin2018high}, Theorem 2.8.1.

In Theorem \ref{Theorem subexponential} both the variance-proxy and the
scale proxy depend on the sub-exponential norms $\left\Vert \cdot
\right\Vert _{\psi _{1}}$. A well known two-tailed bound for sums of bounded
variables, Bernstein's inequality \cite{McDiarmid98}, has the variance proxy
depending on $\left\Vert \cdot \right\Vert _{2}$ and the scale-proxy on $%
\left\Vert \cdot \right\Vert _{\infty }$. When $\left\Vert \cdot \right\Vert
_{2}\ll \left\Vert \cdot \right\Vert _{\infty }$ this leads to tighter
bounds, whenever the inequality is operating in the sub-Gaussian regime,
which often happens for large sample-sizes. The next result allows a similar
use, whenever $\left\Vert \cdot \right\Vert _{2p}\ll q\left\Vert \cdot
\right\Vert _{\psi _{1}}$ for conjugate exponents $p$ and $q$.

\begin{theorem}
	\label{Theorem Bernsteinoid}With $f$ and $X$ as above let $p,q\in \left(
	1,\infty \right) $ satisfy $p^{-1}+q^{-1}=1$. Then for any $t>0$%
	\begin{equation*}
	\Pr \left\{ f\left( X\right) -\mathbb{E}\left[ f\left( X^{\prime }\right) %
	\right] >t\right\} \leq \exp \left( \frac{-t^{2}}{2\left\Vert
		\sum_{k}\left\Vert f_{k}\left( X\right) \right\Vert _{2p}^{2}\right\Vert
		_{\infty }+2eq\max_{k}\left\Vert \left\Vert f_{k}\left( X\right) \right\Vert
		_{\psi _{1}}\right\Vert _{\infty }t}\right) .
	\end{equation*}
\end{theorem}

We cannot let $p\rightarrow 1$ to recover the behaviour of Bernstein's
inequality in the sub-Gaussian regime, because this would drive the
scale-proxy to infinity. But already $p=q=2$ can give substantial
improvements over Theorem \ref{Theorem subexponential}, if the distributions
of the $f_{k}\left( X\right) $ are very concentrated. This inequality
appears to be new even if applied to sums. A proof is given in the
supplement, where we also show, that the $q$ in the scale-proxy can be
replaced by $\sqrt{q}$, if the sub-exponential norm is replaced by the
sub-Gaussian norm.\bigskip 

We conclude this section with a centering lemma, which will be useful in
applications. The proof is given in the supplement.

\begin{lemma}
	\label{Lemma conditional and centering} Let $X,X^{\prime }$ be iid with
	values in $\mathcal{X}$, $\phi :\mathcal{X\times X\rightarrow \mathbb{R}}$
	measurable, $\alpha \in \left\{ 1,2\right\} $. Then
	
	(i) $\left\Vert \mathbb{E}\left[ \phi \left( X,X^{\prime }\right) |X\right]
	\right\Vert _{\psi _{\alpha }}\leq \left\Vert \phi \left( X,X^{\prime
	}\right) \right\Vert _{\psi _{\alpha }}$
	
	(ii) If $\mathcal{X=\mathbb{R}}$ then $\left\Vert X-\mathbb{E}\left[ X\right]
	\right\Vert _{\psi _{\alpha }}\leq 2\left\Vert X\right\Vert _{\psi _{\alpha
		}}$.
	\end{lemma}
	
	One consequence of this lemma is, that we could equally well work with
	uncentered conditional versions, if we adjust the constants by an additional
	factor of $2$.
	
	\section{Applications\label{Section Applications}}
	
	To illustrate the use of these inequalities we give applications to vector
	valued concentration and different methods to prove generalization bounds.
	We concentrate mainly on applications of the more novel Theorems \ref%
	{Theorem subexponential} and \ref{Theorem Bernsteinoid}. Applications of the
	the sub-Gaussian inequality can often be substituted by the reader following
	the same pattern.
	
	\subsection{The sub-exponential norm}
	
	As all our applications use the sub-exponential norm we make some
	explanatory remarks before coming to the applications proper.
	Sub-exponential variables ($\left\Vert Z\right\Vert _{\psi _{1}}<\infty $)
	have {heavier} tails than sub-Gaussian variables and include the
	exponential, chi-squared and Poisson distributions. Products and squares of
	sub-Gaussian variables are sub-exponential, in particular 
	\begin{equation*}
	\left\Vert Z^{2}\right\Vert _{\psi _{1}}=\sup_{p\geq 1}\frac{\left\Vert
		Z^{2}\right\Vert _{p}}{p}=2\sup_{p\geq 1}\left( \frac{\left\Vert
		Z\right\Vert _{2p}}{\sqrt{2p}}\right) ^{2}\leq 2\left\Vert Z\right\Vert
	_{\psi _{2}}^{2}
	\end{equation*}%
	(we would have $\left\Vert Z^{2}\right\Vert _{\psi _{1}}=\left\Vert
	Z\right\Vert _{\psi _{2}}^{2}$, if the norms were defined as in \cite%
	{vershynin2018high}). All sub-Gaussian and bounded variables are
	sub-exponential. For bounded variables we have $\left\Vert Z\right\Vert
	_{\psi _{1}}\leq \left\Vert Z\right\Vert _{\psi _{2}}\leq \left\Vert
	Z\right\Vert _{\infty }$, but for concentrated variables the sub-Gaussian
	and sub-exponential norms can be much smaller. The arithmetic mean of $N$
	iid bounded variables has uniform norm $O\left( 1\right) $, sub-Gaussian
	norm $O\left( N^{-1/2}\right) $, and the square of the mean has
	sub-exponential norm $O\left( N^{-1}\right) $ (see \cite{vershynin2018high}%
	). Our inequalities can therefore be applied successfully to $n$ such
	variables, even when the bounded difference inequality gives only trivial
	results, for example when $n<\ln \left( 1/\delta \right) $, where $\delta $
	is the confidence parameter. For strongly concentrated variables we have the
	following lemma (with proof in the supplement).
	
	\begin{lemma}
		\label{Lemma subexponential norm and concentration}Suppose the random
		variable $X$ satisfies $E\left[ X\right] =0$, $\left\vert X\right\vert \leq 1
		$ a.s. and $\Pr \left\{ \left\vert X\right\vert >\epsilon \right\} \leq
		\epsilon $ for some $\epsilon >0$. Then $\forall p\geq 1,\left\Vert
		X\right\Vert _{p}\leq 2\epsilon ^{1/p}$ and $\left\Vert X\right\Vert _{\psi
			_{1}}\leq 2\left( e\ln \left( 1/\epsilon \right) \right) ^{-1}$.
	\end{lemma}
	
	In a nearly deterministic situation, with $\epsilon =e^{-d}$, we have $%
	\left\Vert X\right\Vert _{\psi _{1}}\leq O\left( 1/d\right) $, and a simple
	union bound of the sub-exponential inequalities allows uniform estimation of 
	$e^{d}$ such variables with sample size $O\left( n\right) \leq O\left(
	d\right) $.
	
	In several places we will require a sub-Gaussian or sub-exponential bound on
	the norm of random vectors. This may seem quite restrictive. If $X=\sum Z_{i}
	$, then in general we can only say $\left\Vert \left\Vert X\right\Vert
	\right\Vert _{\psi _{1}}\leq \sum_{i}\left\Vert \left\Vert Z_{i}\right\Vert
	\right\Vert _{\psi _{1}}$, so if $\mathcal{X}=\mathbb{R}^{d}$ with basis $%
	\left( e_{i}\right) $ then our most general estimate is $\left\Vert
	\left\Vert X\right\Vert \right\Vert _{\psi _{1}}\leq
	\sum_{i=1}^{d}\left\Vert \left\langle e_{i},X\right\rangle \right\Vert
	_{\psi _{1}}$, which has poor dimension dependence. But in many situations
	in machine learning one can assume that $X$ is a sum, $X=Z_{signal}+Z_{noise}
	$, where $\left\Vert Z_{signal}\right\Vert $ is bounded and the perturbing
	component $\left\Vert Z_{noise}\right\Vert $ is of small sub-exponential
	norm, albeit potentially unbounded.
	
	\subsection{Vector valued concentration\label{Subsection vector valued}}
	
	We begin with concentration of norms in a normed space $\left( \mathcal{X}%
	,\left\Vert .\right\Vert \right) $. 
	
	\begin{proposition}
		\label{Proposition vector concentration}Suppose the $X_{i}$ are independent
		random variables with values in a normed space $\left( \mathcal{X}%
		,\left\Vert .\right\Vert \right) $ such that $\left\Vert \left\Vert
		X_{i}\right\Vert \right\Vert _{\psi _{1}}\leq \infty $ and that $\delta >0$.
		(i) With probability at least $1-\delta $%
		\begin{equation*}
		\left\Vert \sum_{i}X_{i}\right\Vert -\mathbb{E}\left\Vert
		\sum_{i}X_{i}\right\Vert \leq 4e\sqrt{\sum_{k}\left\Vert \left\Vert
			X_{k}\right\Vert \right\Vert _{\psi _{1}}^{2}\ln \left( 1/\delta \right) }%
		+4e\max_{k}\left\Vert \left\Vert X_{k}\right\Vert \right\Vert _{\psi
			_{1}}\ln \left( 1/\delta \right) .
		\end{equation*}%
		The inequality is two-sides, that is the two terms on the left-hand-side may
		be interchanged.
		
		(ii) If $\mathcal{X}$ is a Hilbert space, the $X_{i}$ are iid, $n\geq \ln
		\left( 1/\delta \right) \geq \ln 2$, then with probability at least $%
		1-\delta $%
		\begin{equation}
		\left\Vert \frac{1}{n}\sum_{i}X_{i}-\mathbb{E}\left[ X_{1}^{\prime }\right]
		\right\Vert \leq 8e\left\Vert \left\Vert X_{1}\right\Vert \right\Vert _{\psi
			_{1}}\sqrt{\frac{2\ln \left( 1/\delta \right) }{n}}.
		\label{Hilbert space vector concentration}
		\end{equation}
		
		(iii) If $\mathcal{X}$ is a Hilbert space, the $X_{i}$ are iid, $0<\delta
		\leq 1/2$ and $p,q\in \left( 1,\infty \right) $ are conjugate exponents then
		with probability at least $1-\delta $%
		\begin{equation*}
		\left\Vert \frac{1}{n}\sum_{i}X_{i}-\mathbb{E}\left[ X_{1}^{\prime }\right]
		\right\Vert \leq 2\left\Vert \left\Vert X_{1}-\mathbb{E}\left[ X_{1}^{\prime
		}\right] \right\Vert \right\Vert _{2p}\sqrt{\frac{2\ln \left( 1/\delta
			\right) }{n}}+4eq\left\Vert \left\Vert X_{1}\right\Vert \right\Vert _{\psi
		_{1}}\frac{\ln \left( 1/\delta \right) }{n}.
	\end{equation*}
\end{proposition}

The purpose of the simple inequality (\ref{Hilbert space vector
	concentration}) is to give a compact expression, when it is possible to
restrict to the sub-Gaussian regime with the assumption $n\geq \ln \left(
1/\delta \right) $. This is often possible in applications. Part (iii)\
gives better estimation bounds when the lower order moments $\left\Vert
\cdot \right\Vert _{2p}$ ar small.

\begin{proof}
	(i) We look at the function $f\left( x\right) =\left\Vert
	\sum_{i}x_{i}\right\Vert $. Then 
	\begin{equation*}
	\left\vert f_{k}\left( X\right) \left( x\right) \right\vert =\left\vert
	\left\Vert \sum_{i\neq k}x_{i}+X_{k}\right\Vert -\mathbb{E}\left[ \left\Vert
	\sum_{i\neq k}x_{i}+X_{k}^{\prime }\right\Vert \right] \right\vert \leq 
	\mathbb{E}\left[ \left\Vert X_{k}-X_{k}^{\prime }\right\Vert |X\right] .
	\end{equation*}%
	Observe that the bound on $f_{k}\left( X\right) \left( x\right) $ (not $%
	f_{k}\left( X\right) \left( x\right) $ itself) is independent of $x$. Using
	Lemma \ref{Lemma conditional and centering} we get%
	\begin{equation*}
	\left\Vert f_{k}\left( X\right) \left( x\right) \right\Vert _{\psi _{1}}\leq
	\left\Vert \left\Vert X_{k}-X_{k}^{\prime }\right\Vert \right\Vert _{\psi
		_{1}}\leq 2\left\Vert \left\Vert X_{k}\right\Vert \right\Vert _{\psi _{1}},
	\end{equation*}%
	and the first conclusion follows from Theorem \ref{Theorem subexponential}
	by equating the probability to $\delta $ and solving for $t$. The proofs of
	(ii) and (iii) follow a similar pattern and are given in the supplement.
\end{proof}

\subsection{A uniform bound for PSA}

With the results of the previous section it is very easy to obtain a uniform
bound for principal subspace selection (PSA, sometimes PCA is used instead)
with sub-Gaussian data. In PSA we look for a projection onto a $d$%
-dimensional subspace which most faithfully represents the data. Let $H$ be
a Hilbert-space, $X_{i}$ iid with values in $H$ and $\mathcal{P}_{d}$ the
set of $d$-dimensional orthogonal projection operators in $H$. For $x\in H$
and $P\in \mathcal{P}_{d}$ the reconstruction error is $\ell \left(
P,x\right) :=\left\Vert Px-x\right\Vert _{H}^{2}$. We give a bound on the
estimation difference between the expected and the empirical reconstruction
error, uniform for projections in $\mathcal{P}_{d}$.

\begin{theorem}
	With $X=\left( X_{1},...,X_{n}\right) $ iid and $n\geq \ln \left( 1/\delta
	\right) \geq \ln 2$ we have with probability at least $1-\delta $%
	\begin{equation*}
	\sup_{P\in \mathcal{P}_{d}}\frac{1}{n}\sum_{i}\mathbb{E}\left[ \ell \left(
	P,X_{1}\right) \right] -\ell \left( P,X_{i}\right) \leq 16e\left( \sqrt{d}%
	+1\right) \left\Vert \left\Vert X_{1}\right\Vert \right\Vert _{\psi _{2}}^{2}%
	\sqrt{\frac{2\ln \left( 2/\delta \right) }{n}}.
	\end{equation*}
\end{theorem}

\begin{proof}
	It is convenient to work in the space of Hilbert-Schmidt operators $HS\left(
	H\right) $, where we can write $\ell \left( P,x\right) =\left\Vert
	Q_{x}\right\Vert _{HS}-\left\langle P,Q_{x}\right\rangle _{HS}$. Then%
	\begin{multline*}
	\sup_{P\in \mathcal{P}_{d}}\frac{1}{n}\sum_{i}\mathbb{E}\left[ \ell \left(
	P,X_{1}\right) \right] -\ell \left( P,X_{i}\right)  \\
	=\sup_{P\in \mathcal{P}_{d}}\left\langle P,\frac{1}{n}\sum_{i}\left(
	Q_{X_{i}}-\mathbb{E}\left[ Q_{X_{i}}\right] \right) \right\rangle
	_{HS}+\left( \mathbb{E}\left\Vert Q_{X_{i}}\right\Vert _{HS}-\frac{1}{n}%
	\sum_{i}\left\Vert Q_{X_{i}}\right\Vert _{HS}\right) .
	\end{multline*}%
	Since for $P\in \mathcal{P}_{d}$ we have $\left\Vert P\right\Vert _{HS}=%
	\sqrt{d}$, we can use Cauchy-Schwarz and Proposition \ref{Proposition vector
		concentration} (ii) to bound the first term above with probability at least $%
	1-\delta $ as 
	\begin{equation*}
	\sqrt{d}\left\Vert \frac{1}{n}\sum_{i}\left( Q_{X_{i}}-\mathbb{E}\left[
	Q_{X_{1}}\right] \right) \right\Vert _{HS}\leq 8e\sqrt{d}\left\Vert
	\left\Vert Q_{X_{1}}\right\Vert _{HS}\right\Vert _{\psi _{1}}\sqrt{\frac{%
			2\ln \left( 1/\delta \right) }{n}}.
	\end{equation*}%
	The remaining term is bounded by applying the same result to the random
	vectors $\left\Vert Q_{X_{i}}\right\Vert _{HS}$ in the Hilbert space $%
	\mathbb{R}$ (note that this just involves a sum of sub-exponential variables
	and could already be handled with Theorem 2.8.1 in \cite{vershynin2018high}%
	). The result follows from combining both bounds in a union bound and noting
	that $\left\Vert \left\Vert Q_{X_{1}}\right\Vert _{HS}\right\Vert _{\psi
		_{1}}=\left\Vert \left\Vert X_{1}\right\Vert ^{2}\right\Vert _{\psi
		_{1}}\leq 2\left\Vert \left\Vert X_{1}\right\Vert \right\Vert _{\psi
		_{2}}^{2}$.
\end{proof}

\subsection{Generalization with Rademacher complexities}

Suppose that $\mathcal{H}$ is a class of functions $h:\mathcal{X}\rightarrow 
\mathbb{R}$. We seek a high-probability bound on the random variable%
\begin{equation*}
f\left( X\right) =\sup_{h\in \mathcal{H}}\frac{1}{n}\sum_{i=1}^{n}\left(
h\left( X_{i}\right) -\mathbb{E}\left[ h\left( X_{i}^{\prime }\right) \right]
\right) .
\end{equation*}%
The now classical method of Rademacher complexities (\cite{Bartlett02}, \cite%
{Koltchinskii00}) writes $f\left( X\right) $ as the sum 
\begin{equation}
f\left( X\right) =\left( f\left( X\right) -\mathbb{E}\left[ f\left(
X^{\prime }\right) \right] \right) +\mathbb{E}\left[ f\left( X^{\prime
}\right) \right]   \label{Rademacher decomposition}
\end{equation}%
and bounds the two terms separately. The first term is bounded using a
concentration inequality, the second term $\mathbb{E}\left[ f\left( X\right) %
\right] $ is bounded by symmetrization. If the $\epsilon _{i}$ are
independent Rademacher variables, uniformly distributed on $\left\{
-1,1\right\} ,$ then 
\begin{equation*}
\mathbb{E}\left[ f\left( X\right) \right] \leq \mathbb{E}\left[ \frac{2}{n}%
\mathbb{E}\left[ \sup_{h\in \mathcal{H}}\sum_{i}\epsilon _{i}h\left(
X_{i}\right) |X\right] \right] =:\mathbb{E}\left[ \mathcal{R}\left( \mathcal{%
	H},X\right) \right] .
\end{equation*}%
Further bounds on this quantity depend on the class in question, but they do
not necessarily require the $h\left( X_{i}\right) $ to be bounded random
variables, Lipschitz properties being more relevant. For the first term $%
f\left( X\right) -\mathbb{E}\left[ f\left( X\right) \right] $, however, the
classical approach uses the bounded difference inequality, which requires
boundedness. We now show that boundedness can be replaced by sub-exponential
distributions for uniformly Lipschitz function classes.

\begin{theorem}
	\label{Theorem Rademacher application}Let $X=\left( X_{1},...,X_{n}\right) $
	be iid random variables with values in a Banach space $\left( \mathcal{X}%
	,\left\Vert \cdot \right\Vert \right) $ and let $\mathcal{H}$ be a class of
	functions $h:\mathcal{X}\rightarrow \mathbb{R}$ such that $h\left( x\right)
	-h\left( y\right) \leq L\left\Vert x-y\right\Vert $ for all $h\in \mathcal{H}
	$ and $x,y\in \mathcal{X}$. If $n\geq \ln \left( 1/\delta \right) $ then
	with probability at least $1-\delta $%
	\begin{equation*}
	\sup_{h\in \mathcal{H}}\frac{1}{n}\sum_{i}h\left( X_{i}\right) -\mathbb{E}%
	\left( h\left( X\right) \right) \leq \mathbb{E}\left[ \mathcal{R}\left( 
	\mathcal{H},X\right) \right] +16eL\left\Vert \left\Vert X_{1}\right\Vert
	\right\Vert _{\psi _{1}}\sqrt{\frac{\ln \left( 1/\delta \right) }{n}}.
	\end{equation*}
\end{theorem}

\begin{proof}
	The vector space 
	\begin{equation*}
	\mathcal{B}=\left\{ g:\mathcal{H}\rightarrow \mathbb{R}:\sup_{h\in \mathcal{H%
		}}\left\vert g\left( h\right) \right\vert <\infty \right\} 
		\end{equation*}%
		becomes a normed space with norm $\left\Vert g\right\Vert _{\mathcal{B}%
		}=\sup_{h\in \mathcal{H}}\left\vert g\left( h\right) \right\vert $\textbf{\ }%
		. For each $X_{i}$ define $\hat{X}_{i}$\textbf{\ }$\in \mathcal{B}$ by $\hat{%
			X}_{i}\left( h\right) =\left( 1/n\right) \left( h\left( X_{i}\right) -%
		\mathbb{E}\left[ h\left( X_{i}^{\prime }\right) \right] \right) $. Then the $%
		\hat{X}_{i}$ are zero mean random variables in $\mathcal{B}$ and $f\left(
		X\right) =\left\Vert \sum_{i}\hat{X}_{i}\right\Vert _{\mathcal{B}}$. Also
		with Lemma \ref{Lemma conditional and centering} and the iid-assumption%
		\begin{eqnarray*}
			\left\Vert \left\Vert \hat{X}_{i}\right\Vert _{\mathcal{B}}\right\Vert
			_{\psi _{\alpha }} &=&\frac{1}{n}\left\Vert \sup_{h}\left( \mathbb{E}\left[
			h\left( X_{i}\right) -h\left( X_{i}^{\prime }\right) \right] |X\right)
			\right\Vert _{\psi _{\alpha }} \\
			&\leq &\frac{L}{n}\left\Vert \mathbb{E}\left[ \left\Vert X_{i}-X_{i}^{\prime
			}\right\Vert \right] |X\right\Vert _{\psi _{\alpha }}\leq \frac{2L}{n}%
			\left\Vert \left\Vert X_{1}\right\Vert \right\Vert _{\psi _{\alpha }},
		\end{eqnarray*}%
		and from Proposition \ref{Proposition vector concentration} (ii) we get with
		probability at least $1-\delta $%
		\begin{equation*}
		f\left( X\right) -\mathbb{E}\left[ f\left( X^{\prime }\right) \right] \leq
		16eL\left\Vert \left\Vert X_{1}\right\Vert \right\Vert _{\psi _{1}}\sqrt{%
			\frac{\ln \left( 1/\delta \right) }{n}}.
		\end{equation*}%
		The result follows from (\ref{Rademacher decomposition}).
	\end{proof}
	
	\textbf{Remarks.} 1. A possible candidate for $\mathcal{H}$ would be a ball
	of radius $L$ in the dual space $\mathcal{X}^{\ast }$, composed with
	Lipschitz functions, like the hinge-loss.
	
	2. If, instead of using Proposition \ref{Proposition vector concentration},
	one directly considers the centered conditional versions of $f\left(
	X\right) $, the constants above can be improved at the expense of a slightly
	more complicated proof. 
	
	3. A corresponding sub-Gaussian result can be supplied along the same lines
	by using Theorem \ref{Theorem subgaussian} instead of Theorem \ref{Theorem
		subexponential}. Such a result has been given in \cite{Meir03}, Theorem 3,
	using a sub-Gaussian condition which involves the supremum over the function
	class. The sub-exponential bound above is new as far as we know, and in the
	relevant regime $n>\ln \left( 1/\delta \right) $ it improves over the
	sub-Gaussian case, since $\left\Vert \left\Vert X_{1}\right\Vert \right\Vert
	_{\psi _{1}}\leq \left\Vert \left\Vert X_{1}\right\Vert \right\Vert _{\psi
		_{2}}$.
	
	As a concrete case consider linear regression with potentially unbounded
	data. Let $\mathcal{X}=\left( H,\mathbb{R}\right) $, where $H$ is a
	Hilbert-space with inner product $\left\langle .,.\right\rangle $ and norm $%
	\left\Vert .\right\Vert _{H}$, and let $X_{1}$ and $Z_{1}$ be each
	sub-exponential random variables in $H$ and $\mathbb{R}$ respectively. The
	pair $\left( X_{1},Z_{1}\right) $ represents the joint occurrence of
	input-vectors $X_{1}$ and real outputs $Z_{1}$. On $\mathcal{X}$ we consider
	the class $\mathcal{H}$ of functions $\mathcal{H}=\left\{ \left( x,z\right)
	\mapsto h\left( x,z\right) =\ell \left( \left\langle w,x\right\rangle
	-z\right) :\left\Vert w\right\Vert _{H}\leq L\right\} $, where $\ell $ is a $%
	1$-Lipschitz loss function, like the absolute error or the Huber loss.
	
	\begin{corollary}
		\label{Corollary regression}Let $\mathcal{X}$ and $\mathcal{H}$ be as above
		and $\left( X,Z\right) =\left( \left( X_{1},Z_{1}\right) ,...,\left(
		X_{n},Z_{n}\right) \right) $ be an iid sample of random variables in $%
		\mathcal{X}$. Then for $\delta >0$ and $n\geq \ln \left( 1/\delta \right) $ 
		\begin{equation*}
		\sup_{h\in \mathcal{H}}\frac{1}{n}\sum_{i}h\left( X_{i},Z_{i}\right) -%
		\mathbb{E}\left( h\left( X_{i},Z_{i}\right) \right) \leq \frac{8}{\sqrt{n}}%
		\left( L\left\Vert \left\Vert X_{1}\right\Vert \right\Vert _{\psi
			_{1}}+\left\Vert \left\Vert Z_{1}\right\Vert \right\Vert _{\psi _{1}}\right)
		\left( 1+2e\sqrt{\ln \left( 1/\delta \right) }\right) .
		\end{equation*}
	\end{corollary}
	
	The proof of this corollary is given in the supplement.
	
	\subsection{Unbounded metric spaces and algorithmic stability}
	
	We use Theorem 2 to extend a method of Kontorovich \cite%
	{kontorovich2014concentration} from sub-Gaussian to sub-exponential
	distributions. If $\left( \mathcal{X},d,\mu \right) $ is a metric
	probability space and $X,X^{\prime }\sim \mu $ are iid random variables with
	values in $\mathcal{X}$, Kontorovich defines the sub-Gaussian diameter of $%
	\left( \mathcal{X},d\right) $ as the optimal sub-Gaussian parameter of the
	random variable $\epsilon d\left( X,X^{\prime }\right) $, where $\epsilon $
	is a Rademacher variable. The Rademacher variable is needed in \cite%
	{kontorovich2014concentration} to work with centered random variables, which
	gives better constants. In our case we work with norms and we can more
	simply define the sub-Gaussian and sub-exponential diameters respectively as 
	\begin{equation*}
	\Delta _{\alpha }\left( \mathcal{X},d,\mu \right) =\left\Vert d\left(
	X,X^{\prime }\right) \right\Vert _{\psi _{\alpha }}\text{ for }\alpha \in
	\left\{ 1,2\right\} \text{ and independet }X^{\prime },X\sim \mu \text{. }
	\end{equation*}%
	Then Theorem \ref{Theorem subexponential} implies the following result, the
	easy proof of which is given in the supplement.
	
	\begin{theorem}
		\label{Theorem metric}For $1\leq i\leq n$ let $X_{i}$ be independent random
		variables distributed as $\mu _{i}$ in $\mathcal{X}$, $X=\left(
		X_{1},...,X_{n}\right) $, $X^{\prime }$ iid to $X$, and let $f:\mathcal{X}%
		^{n}\rightarrow \mathbb{R}$ have Lipschitz constant $L$ with respect to the
		metric $\rho $ on $\mathcal{X}^{n}$ defined by $\rho \left( x,y\right)
		=\sum_{i}d\left( x_{i},y_{i}\right) $. Then for $t>0$%
		\begin{equation*}
		\Pr \left\{ f\left( X\right) -\mathbb{E}\left[ f\left( X^{\prime }\right) %
		\right] >t\right\} \leq \exp \left( \frac{-t^{2}}{4eL^{2}\sum_{k}\Delta
			_{1}\left( \mathcal{X},d,\mu _{i}\right) ^{2}+2e\max_{i}\Delta _{1}\left( 
			\mathcal{X},d,\mu _{i}\right) t}\right) .
		\end{equation*}
	\end{theorem}
	
	This is the sub-exponential counterpart to Theorem 1 of \cite%
	{kontorovich2014concentration}, a version of which could have been derived
	using Theorem \ref{Theorem subgaussian} in place of \ref{Theorem
		subexponential}. Our result can be equally substituted to establish
	generalization using the notion of total Lipschitz stability, just as in 
	\cite{kontorovich2014concentration}. We also note that Theorem 4 of the
	latter work also gives bounds for different Orlicz norms $\left\Vert
	.\right\Vert _{\psi _{p}}$, but it requires $p>1$, and the bounds
	deteriorate as $p\rightarrow 1$.
	
	\section{Proofs of Theorems \protect\ref{Theorem subgaussian} and \protect
		\ref{Theorem subexponential}\label{Section Proofs}}
	
	We first collect some necessary tools. Central to the entropy method is the
	entropy $S\left( Y\right) $ of a real valued random variable $Y$ defined as%
	\begin{equation*}
	S\left( Y\right) =\mathbb{E}_{Y}\left[ Y\right] -\ln \mathbb{E}\left[ e^{Y}%
	\right] ,
	\end{equation*}%
	where the tilted expectation $\mathbb{E}_{Y}$ is defined as $\mathbb{E}_{Y}%
	\left[ Z\right] =\mathbb{E}\left[ Ze^{Y}\right] /\mathbb{E}\left[ e^{Y}%
	\right] $. The logarithm of the moment generating function can be expressed
	in terms of the entropy as%
	\begin{equation}
	\ln \mathbb{E}\left[ e^{\beta \left( Y-\mathbb{E}\left[ Y\right] \right) }%
	\right] =\beta \int_{0}^{\beta }\frac{S\left( \gamma Y\right) d\gamma }{%
		\gamma ^{2}}  \label{Log moment generating function bound}
	\end{equation}%
	(Theorem 1 in \cite{Maurer12}). If $f:\mathcal{X}^{n}\rightarrow \mathbb{R}$
	and $X$ and the $f_{k}$ are as in the introduction then the conditional
	entropy is the function $S_{f,k}:\mathcal{X}^{n}\rightarrow \mathbb{R}$
	defined by $S_{f,k}\left( x\right) =S\left( f_{k}\left( X\right) \left(
	x\right) \right) $ for $x\in \mathcal{X}^{n}$. At the heart of the method is
	the sub-additivity of entropy (Theorem 6 in \cite{Maurer12} or Theorem 4.22
	in \cite{Boucheron13})
	
	\begin{equation}
	S\left( f\left( X\right) \right) \leq \mathbb{E}_{f\left( X\right) }\left[
	\sum_{i=1}^{n}S_{f,k}\left( X\right) \right] .
	\label{Subadditivity of entropy}
	\end{equation}%
	The following lemma gives a bound on the entropy of a sub-Gaussian random
	variable.
	
	\begin{lemma}
		\label{Lemma sub-Gaussian entropy bound}(i) for any random variables $Y$ we
		have $S\left( Y\right) \leq \ln \mathbb{E}\left[ e^{2Y}\right] $. (ii) If $Y$
		is sub-Gaussian and $\beta $ is real then $S\left( \beta Y\right) \leq
		16e\beta ^{2}\left\Vert Y\right\Vert _{\phi _{2}}^{2}.$
	\end{lemma}
	
	\begin{proof}
		Since $S\left( Y\right) =S\left( Y-\mathbb{E}\left[ Y\right] \right) $, we
		can assume $Y$ centered.%
		\begin{eqnarray*}
			S\left( Y\right)  &=&\mathbb{E}_{Y}\left[ \ln \left( \frac{e^{Y}}{\mathbb{E}%
				\left[ e^{Y}\right] }\right) \right] \leq \ln \mathbb{E}_{Y}\left[ \frac{%
				e^{Y}}{\mathbb{E}\left[ e^{Y}\right] }\right] =\ln \mathbb{E}\left[ e^{2Y}%
			\right] -2\ln \mathbb{E}\left[ e^{Y}\right]  \\
			&\leq &\ln \mathbb{E}\left[ e^{2Y}\right] .
		\end{eqnarray*}%
		The first inequality follows from Jensen's inequality by concavity of the
		logarithm, the second by convexity of the exponential function. This gives
		(i). For (ii) replace $Y$ by $\beta Y$ and use (\ref{Subgaussian MGF bound})
		to get $S\left( \beta Y\right) \leq \ln \mathbb{E}\left[ e^{2\beta Y}\right]
		\leq 16e\beta ^{2}\left\Vert Y\right\Vert _{\psi _{2}}^{2}.$
	\end{proof}
	
	\begin{proof}[\textbf{Proof of Theorem \protect\ref{Theorem subgaussian}}]
		For any $x\in \mathcal{X}^{n}$ and $\gamma \in \mathbb{R}$ part (ii) of the
		previous lemma gives $S_{\gamma f,k}\left( x\right) =S\left( \gamma
		f_{k}\left( X\right) \left( x\right) \right) \leq 16e\gamma ^{2}\left\Vert
		f_{k}\left( X\right) \left( x\right) \right\Vert _{\psi _{2}}^{2}$. By
		subadditivity of entropy (\ref{Subadditivity of entropy}) this gives%
		\begin{equation}
		S\left( \gamma f\left( X\right) \right) \leq 16e\gamma ^{2}\mathbb{E}%
		_{\gamma f\left( X\right) }\left[ \sum_{k}\left\Vert f_{k}\left( X^{\prime
		}\right) \right\Vert _{\psi _{2}}^{2}\left( X\right) \right] \leq 16e\gamma
		^{2}\left\Vert \sum_{k}\left\Vert f_{k}\left( X\right) \right\Vert _{\psi
			_{2}}^{2}\right\Vert _{\infty }.  \notag
		\end{equation}%
		Using Markov's inequality and (\ref{Log moment generating function bound})
		this gives%
		\begin{eqnarray*}
			\Pr \left\{ f\left( X\right) -\mathbb{E}\left[ f\left( X^{\prime }\right) %
			\right] >t\right\}  &\leq &\exp \left( \beta \int_{0}^{\beta }\frac{S\left(
				\gamma f\left( X\right) \right) d\gamma }{\gamma ^{2}}-\beta t\right)  \\
			&\leq &\exp \left( 16e\beta ^{2}\left\Vert \sum_{k}\left\Vert f_{k}\left(
			X\right) \right\Vert _{\psi _{2}}^{2}\right\Vert _{\infty }d\gamma -\beta
			t\right) .
		\end{eqnarray*}%
		Minimization in $\beta $ concludes the proof.\bigskip 
	\end{proof}
	
	Lemma \ref{Lemma sub-Gaussian entropy bound} (i) and the preceeding proof
	provide a general template to convert many exponential tail-bounds for sums
	into analogous bounds for general functions. For sums $\sum X_{i}$ one
	typically has a bound on $\ln \mathbb{E}\left[ e^{\beta X_{i}}\right] $.
	Lemma \ref{Lemma sub-Gaussian entropy bound} then provides an analogous
	bound on the entropy of the conditional versions of a general function, and
	subadditivity of entropy and (\ref{Log moment generating function bound})
	complete the conversion, albeit with a deterioration of constants. Using
	part (v) of Proposition 2.7.1 in \cite{vershynin2018high} this method would
	lead to an easy proof of Theorem \ref{Theorem subexponential}, in a form
	exactly like Theorem 2.8.1 \cite{vershynin2018high}$.$ Here we will use a
	slightly different method which gives better constants and will also provide
	the proof of Theorem \ref{Theorem Bernsteinoid}.
	
	For the proof of Theorem \ref{Theorem subexponential} we use the following
	fluctuation representation of entropy (Theorem 3 in \cite{Maurer12}).%
	\begin{equation}
	S\left( Y\right) =\int_{0}^{1}\left( \int_{t}^{1}\mathbb{E}_{sY}\left[
	\left( Y-\mathbb{E}_{sY}\left[ Y\right] \right) ^{2}\right] ds\right) ~dt
	\label{Fluctuation representation}
	\end{equation}%
	We use this to bound the entropy of a centered sub-exponential random
	variable.
	
	\begin{lemma}
		\label{Lemma subexponential entropy bound}If $\left\Vert Y\right\Vert _{\psi
			_{1}}<1/e$ and $\mathbb{E}\left[ Y\right] =0$ then%
		\begin{equation*}
		S\left( Y\right) \leq \frac{e^{2}\left\Vert Y\right\Vert _{\psi _{1}}^{2}}{%
			\left( 1-e\left\Vert Y\right\Vert _{\psi _{1}}\right) ^{2}}.
		\end{equation*}
	\end{lemma}
	
	\begin{proof}
		Let $s\in \left[ 0,1\right] $%
		\begin{equation*}
		\mathbb{E}_{sY}\left[ \left( Y-\mathbb{E}_{sY}\left[ Y\right] \right) ^{2}%
		\right] \leq \mathbb{E}_{sY}\left[ Y^{2}\right] =\frac{\mathbb{E}\left[
			Y^{2}e^{sY}\right] }{\mathbb{E}\left[ e^{sY}\right] }\leq \mathbb{E}\left[
		Y^{2}e^{sY}\right] .
		\end{equation*}%
		The first inequality follows from the variational property of variance, the
		second from Jensen's inequality since $\mathbb{E}\left[ Y_{k}\right] =0$.
		Expanding the exponential we get%
		\begin{equation*}
		\mathbb{E}\left[ Y^{2}e^{sY}\right] \leq \mathbb{E}\left[ \sum_{m=0}^{\infty
		}\frac{s^{m}}{m!}Y^{m+2}\right] =\sum_{m=0}^{\infty }\frac{s^{m}}{m!}\mathbb{%
		E}\left[ Y^{m+2}\right] .
	\end{equation*}%
	The interchange of expectation and summation will be justified by absolute
	convergence of the sum as follows.%
	\begin{eqnarray*}
		\sum_{m=0}^{\infty }\frac{s^{m}}{m!}\mathbb{E}\left[ Y^{m+2}\right]  &\leq
		&\sum_{m=0}^{\infty }\frac{s^{m}}{m!}\left\Vert Y\right\Vert _{\psi
			_{1}}^{m+2}\left( m+2\right) ^{m+2} \\
		&\leq &e^{2}\left\Vert Y\right\Vert _{\psi _{1}}^{2}\sum_{m=0}^{\infty
		}\left( m+2\right) \left( m+1\right) \left( se\left\Vert Y\right\Vert _{\psi
		_{1}}\right) ^{m}.
\end{eqnarray*}%
The first inequality follows from definition of $\left\Vert Y\right\Vert
_{\psi _{1}}$, and the second from Stirling's approximation $\left(
m+2\right) ^{m+2}\leq \left( m+2\right) !e^{m+2}$. Absolute convergence is
insured since $se\left\Vert Y\right\Vert _{\psi _{1}}<1$. Using 
\begin{equation*}
\int_{0}^{1}\int_{t}^{1}s^{m}ds~dt=\frac{1}{m+2}
\end{equation*}%
the fluctuation representation (\ref{Fluctuation representation}) and the
above inequalities give%
\begin{eqnarray*}
	S\left( Y\right)  &=&\int_{0}^{1}\left( \int_{t}^{1}\mathbb{E}_{sY}\left[
	\left( Y-\mathbb{E}_{sY}\left[ Y\right] \right) ^{2}\right] ds\right) ~dt \\
	&\leq &e^{2}\left\Vert Y\right\Vert _{\psi _{1}}^{2}\sum_{m=0}^{\infty
	}\left( m+1\right) \left( e\left\Vert Y\right\Vert _{\phi _{1}}\right) ^{m}=%
	\frac{e^{2}\left\Vert Y\right\Vert _{\psi _{1}}^{2}}{\left( 1-e\left\Vert
		Y\right\Vert _{\psi _{1}}\right) ^{2}}.
\end{eqnarray*}
\end{proof}

We also need the following lemma (Lemma 12 in \cite{maurer2006concentration}%
).

\begin{lemma}
	\label{Lemma Optimization}Let $C$ and $b$ denote two positive real numbers, $%
	t>0$. Then%
	\begin{equation}
	\inf_{\beta \in \lbrack 0,1/b)}\left( -\beta t+\frac{C\beta ^{2}}{1-b\beta }%
	\right) \leq \frac{-t^{2}}{2\left( 2C+bt\right) }.
	\label{False Lemma Inequality 1}
	\end{equation}%
	\bigskip 
\end{lemma}

\begin{proof}[\textbf{Proof of Theorem \protect\ref{Theorem subexponential}}]
	
	We abbreviate $M:=\max_{k}\left\Vert \left\Vert f_{k}\left( X\right)
	\right\Vert _{\psi _{1}}\right\Vert _{\infty }$ and let $0<\gamma \leq \beta
	<\left( eM\right) ^{-1}$. Then for any $x\in \mathcal{X}^{n}$ and $k\in
	\left\{ 1,...,n\right\} $ we have $\left\Vert \gamma f_{k}\left( X\right)
	\left( x\right) \right\Vert _{\psi _{1}}<\left\Vert f_{k}\left( X\right)
	\left( x\right) \right\Vert _{\psi _{1}}/\left( eM\right) \leq 1/e$ by the
	definition of $M$. We can therefore apply the previous lemma to $\gamma
	f_{k}\left( X\right) \left( x\right) $. It gives for almost all $x$%
	\begin{equation*}
	S_{\gamma f,k}\left( x\right) =S\left( \gamma f_{k}\left( X\right) \left(
	x\right) \right) \leq \frac{e^{2}\left\Vert \gamma f_{k}\left( X\right)
		\left( x\right) \right\Vert _{\psi _{1}}^{2}}{\left( 1-e\left\Vert \gamma
		f_{k}\left( X\right) \left( x\right) \right\Vert _{\psi _{1}}\right) ^{2}}%
	\leq \frac{\gamma ^{2}e^{2}\left\Vert f_{k}\left( X\right) \left( x\right)
		\right\Vert _{\psi _{1}}^{2}}{\left( 1-\gamma eM\right) ^{2}}
	\end{equation*}%
	Subadditivity of entropy (\ref{Subadditivity of entropy}) then yields the
	total entropy bound%
	\begin{eqnarray}
	S\left( \gamma f\left( X\right) \right)  &\leq &\mathbb{E}_{\gamma f\left(
		X\right) }\left[ \sum_{k}S_{\gamma f,k}\left( X\right) \right] \leq \frac{%
		\gamma ^{2}e^{2}\mathbb{E}_{\gamma f\left( X\right) }\left[
		\sum_{k}\left\Vert f_{k}\left( X^{\prime }\right) \right\Vert _{\psi
			_{1}}^{2}\left( X\right) \right] }{\left( 1-\gamma eM\right) ^{2}}
	\label{Entropybound subexponential} \\
	&\leq &\frac{\gamma ^{2}e^{2}\left\Vert \sum_{k}\left\Vert f_{k}\left(
		X\right) \right\Vert _{\psi _{1}}^{2}\right\Vert _{\infty }}{\left( 1-\gamma
		eM\right) ^{2}}.  \notag
	\end{eqnarray}%
	Together with (\ref{Log moment generating function bound}) this gives%
	\begin{equation*}
	\ln \mathbb{E}\left[ e^{\beta \left( f-\mathbb{E}f\right) }\right] =\beta
	\int_{0}^{\beta }\frac{S\left( \gamma f\left( X\right) \right) d\gamma }{%
		\gamma ^{2}}\leq \frac{\beta ^{2}e\left\Vert \sum_{k}\left\Vert f_{k}\left(
		X\right) \right\Vert _{\psi _{1}}^{2}\right\Vert _{\infty }}{1-\beta eM},
	\end{equation*}%
	and the concentration inequality then follows from Markov's inequality and
	Lemma \ref{Lemma Optimization}.
\end{proof}

\section{Conclusion}

In this paper, we presented an extension of Hoeffding- and Bernstein-type
inequalities for sums of sub-Gaussian and sub-exponential independent random
variables to general functions, and illustrated these inequalities with
applications to statistical learning theory. 

We hope that future work will reveal other interesting applications of these
inequalities.

\section{Appendix}

\bigskip 

\subsection{Remaining proofs for Section \protect\ref{Section results}}

We give the missing proof for Theorem \ref{Theorem Bernsteinoid}. The next
lemma replaces Lemma \ref{Lemma subexponential entropy bound}.\bigskip 

\begin{lemma}
	Let $\mathbb{E}\left[ Y\right] =0$ and $1<p,q<\infty $ be conjugate
	exponents ($1/p+1/q=1$). If $\left\Vert Y\right\Vert _{\psi _{1}}<1/\left(
	eq\right) $ then%
	\begin{equation*}
	S\left( Y\right) \leq \frac{\left\Vert Y^{2}\right\Vert _{p}}{2\left(
		1-eq\left\Vert Y\right\Vert _{\psi _{1}}\right) ^{2}}.
	\end{equation*}%
	If $\left\Vert Y\right\Vert _{\psi _{2}}<1/\left( e\sqrt{q}\right) $ the
	same inequality holds with $q\left\Vert Y\right\Vert _{\psi _{1}}$ replaced
	by $\sqrt{q}\left\Vert Y\right\Vert _{\psi _{2}}$.
\end{lemma}

\begin{proof}
	As in the proof of Lemma \ref{Lemma subexponential entropy bound} we let $%
	s\in \left[ 0,1\right] $ and obtain the inequality%
	\begin{equation*}
	\mathbb{E}_{sY}\left[ \left( Y-\mathbb{E}_{sY}\left[ Y\right] \right) ^{2}%
	\right] \leq \mathbb{E}\left[ \sum_{m=0}^{\infty }\frac{s^{m}}{m!}Y^{m+2}%
	\right] \leq \left\Vert Y^{2}\right\Vert _{p}\sum_{m=0}^{\infty }\frac{s^{m}%
	}{m!}\left\Vert Y^{m}\right\Vert _{q},
	\end{equation*}%
	where the second bound follows from H\"{o}ler's inequality. Using the
	definition of $\left\Vert .\right\Vert _{\psi _{1}}$ and Stirling's
	approximation give the bound%
	\begin{equation*}
	\left\Vert Y^{m}\right\Vert _{q}=\left\Vert Y\right\Vert _{mq}^{m}\leq
	\left( qm\left\Vert Y\right\Vert _{\psi _{1}}\right) ^{m}=m^{m}\left(
	q\left\Vert Y\right\Vert _{\psi _{1}}\right) ^{m}\leq m!\left( eq\left\Vert
	Y\right\Vert _{\psi _{1}}\right) ^{m},
	\end{equation*}%
	whence, since $\left\Vert Y\right\Vert _{\psi _{1}}<1/\left( eq\right)
	\implies eq\left\Vert Y\right\Vert _{\psi _{1}}<1$, 
	\begin{equation*}
	\sum_{m=0}^{\infty }\frac{s^{m}}{m!}\left\Vert Y^{m}\right\Vert _{q}\leq
	\sum_{m=0}^{\infty }\left( s~eq\left\Vert Y\right\Vert _{\psi _{1}}\right)
	^{m}\leq \sum_{m=0}^{\infty }\left( eq\left\Vert Y\right\Vert _{\psi
		_{1}}\right) ^{m}\leq \frac{1}{1-eq\left\Vert Y\right\Vert _{\psi _{1}}}<%
	\frac{1}{\left( 1-eq\left\Vert Y\right\Vert _{\psi _{1}}\right) ^{2}}.
	\end{equation*}%
	Substitution above gives%
	\begin{equation*}
	\mathbb{E}_{sY}\left[ \left( Y-\mathbb{E}_{sY}\left[ Y\right] \right) ^{2}%
	\right] \leq \left\Vert Y^{2}\right\Vert _{p}\sum_{m=0}^{\infty }\frac{s^{m}%
	}{m!}\left\Vert Y^{m}\right\Vert _{q}<\frac{\left\Vert Y^{2}\right\Vert _{p}%
}{\left( 1-eq\left\Vert Y\right\Vert _{\psi _{1}}\right) ^{2}},
\end{equation*}%
and the double integral in (\ref{Fluctuation representation}) then provides
the factor of $1/2$. For the remaining statement repeat the proof and use%
\begin{equation*}
\left\Vert Y\right\Vert _{mq}^{m}\leq \left( \sqrt{qm}\left\Vert
Y\right\Vert _{\psi _{2}}\right) ^{m}\leq m^{m}\left( \sqrt{q}\left\Vert
Y\right\Vert _{\psi _{2}}\right) ^{m}.
\end{equation*}
\end{proof}

\begin{proof}[\textbf{Proof of Theorem \protect\ref{Theorem Bernsteinoid}}]
	We abbreviate $M:=\max_{k}\left\Vert \left\Vert f_{k}\left( X\right)
	\right\Vert _{\psi _{1}}\right\Vert _{\infty }$ and let $0<\gamma \leq \beta
	<\left( eqM\right) ^{-1}$. Then for any $k\in \left\{ 1,...,n\right\} $ we
	have $\left\Vert \gamma f_{k}\left( X\right) \right\Vert _{\psi
		_{1}}<\left\Vert f_{k}\left( X\right) \right\Vert _{\psi _{1}}/\left(
	eqM\right) \leq 1/\left( eq\right) $ by the definition of $M$. We can
	therefore apply the previous Lemma to the random variable $\gamma
	f_{k}\left( X\right) $. It gives almost surely%
	\begin{equation*}
	S\left( \gamma f_{k}\left( X\right) \right) \leq \frac{\left\Vert \left(
		\gamma f_{k}\left( X\right) \right) ^{2}\right\Vert _{p}}{2\left(
		1-eq\left\Vert \gamma f_{k}\left( X\right) \right\Vert _{\psi _{1}}\right)
		^{2}}\leq \frac{\gamma ^{2}\left\Vert f_{k}\left( X\right) ^{2}\right\Vert
		_{p}}{2\left( 1-\gamma eqM\right) ^{2}}
	\end{equation*}%
	Subadditivity of entropy (\ref{Subadditivity of entropy}) then yields the
	total entropy bound%
	\begin{eqnarray*}
		S\left( \gamma f\left( X\right) \right)  &\leq &\mathbb{E}_{\gamma f}\left[
		\sum_{k}S\left( \gamma f_{k}\left( X\right) \right) \left( X\right) \right]
		\leq \frac{\gamma ^{2}\mathbb{E}_{\gamma f\left( X\right) }\left[
			\sum_{k}\left\Vert f_{k}\left( X\right) ^{2}\right\Vert _{p}\left( X\right) %
			\right] }{2\left( 1-\gamma eqM\right) ^{2}} \\
		&\leq &\frac{\gamma ^{2}\left\Vert \sum_{k}\left\Vert f_{k}\left( X\right)
			^{2}\right\Vert _{p}\right\Vert _{\infty }}{2\left( 1-\gamma eqM\right) ^{2}}%
		.
	\end{eqnarray*}%
	Together with (\ref{Log moment generating function bound}) this gives%
	\begin{equation*}
	\ln \mathbb{E}\left[ e^{\beta \left( f-\mathbb{E}f\right) }\right] =\beta
	\int_{0}^{\beta }\frac{S\left( \gamma f\left( X\right) \right) d\gamma }{%
		\gamma ^{2}}\leq \frac{\beta ^{2}\left\Vert \sum_{k}\left\Vert f_{k}\left(
		X\right) ^{2}\right\Vert _{p}\right\Vert _{\infty }}{2\left( 1-\beta
		eqM\right) },
	\end{equation*}%
	and the concentration inequality then follows from Markov's inequality and
	Lemma \ref{Lemma Optimization}, if we set $C=\left\Vert \sum_{k}\left\Vert
	f_{k}\left( X\right) ^{2}\right\Vert _{p}\right\Vert _{\infty }/2$ and $b=eqM
	$. 
\end{proof}

\begin{proof}[\textbf{Proof of Lemma \protect\ref{Lemma conditional and
			centering}}]
	By Jensen's inequality for $p\geq 1$%
	\begin{eqnarray*}
		\mathbb{E}\left[ \left\vert \mathbb{E}\left[ \phi \left( X,X^{\prime
		}\right) |X\right] \right\vert ^{p}\right]  &\leq &\mathbb{E}\left[ \mathbb{E%
	}\left[ \left\vert \phi \left( X,X^{\prime }\right) \right\vert |X\right]
	^{p}\right] =\mathbb{E}\left[ \mathbb{E}\left[ \left( \left\vert \phi \left(
	X,X^{\prime }\right) \right\vert ^{p}\right) ^{1/p}|X\right] ^{p}\right]  \\
	&\leq &\mathbb{E}\left[ \mathbb{E}\left[ \left\vert \phi \left( X,X^{\prime
	}\right) \right\vert ^{p}|X\right] \right] =\mathbb{E}\left[ \left\vert \phi
	\left( X,X^{\prime }\right) \right\vert ^{p}\right] .
\end{eqnarray*}%
Therefore $\left\Vert \mathbb{E}\left[ \phi \left( X,X^{\prime }\right) |X%
\right] \right\Vert _{p}\leq \left\Vert \phi \left( X,X^{\prime }\right)
\right\Vert _{p}$ and (i) follows from our definition of the two norms. If $%
\mathcal{X=\mathbb{R}}$ and $\phi \left( s,t\right) =s-t$ we get from (i)
that%
\begin{equation*}
\left\Vert X-\mathbb{E}\left[ X^{\prime }\right] \right\Vert _{\psi _{\alpha
	}}=\left\Vert \mathbb{E}\left[ X-X^{\prime }|X\right] \right\Vert _{\psi
	_{\alpha }}\leq \left\Vert X-X^{\prime }\right\Vert _{\psi _{\alpha }}\leq
2\left\Vert X\right\Vert _{\psi _{\alpha }}.
\end{equation*}%
\bigskip 
\end{proof}

\subsection{Remaining proofs for Section \protect\ref{Section Applications}}

\bigskip 

Here is a proof of Lemma \ref{Lemma subexponential norm and concentration}

\begin{proof}[Proof of Lemma \protect\ref{Lemma subexponential norm and
		concentration}]
	For $p\geq 1$ the function $g\left( t\right) =t^{1/p}$ is concave and $%
	g^{\prime }\left( t\right) =\frac{1}{p}t^{\frac{1}{p}-1}$. It follows that
	for $s,t\geq 0$%
	\begin{equation*}
	\left( t+s\right) ^{1/p}\leq t^{1/p}+\frac{st^{\frac{1}{p}-1}}{p}%
	=t^{1/p}\left( 1+\frac{s}{pt}\right) \leq t^{1/p}\left( 1+\frac{s}{t}\right)
	.
	\end{equation*}%
	Under the conditions on $X$ we therefore have 
	\begin{eqnarray*}
		\left\Vert X\right\Vert _{p} &\leq &\left( \epsilon +\epsilon ^{p}\left(
		1-\epsilon \right) \right) ^{1/p}\leq \epsilon ^{1/p}\left( 1+\frac{\epsilon
			^{p}\left( 1-\epsilon \right) }{\epsilon }\right)  \\
		&\leq &\epsilon ^{1/p}\left( 1+\epsilon ^{p-1}\right) \leq 2\epsilon ^{1/p}.
	\end{eqnarray*}%
	This proves the first claim. Calculus shows, that the function $p\mapsto
	2\epsilon ^{1/p}/p$ attains its maximum at $p=\ln \left( 1/\epsilon \right) $%
	, so%
	\begin{equation*}
	\left\Vert X\right\Vert _{\psi _{1}}=\sup_{p\geq 1}\frac{\left\Vert
		X\right\Vert _{p}}{p}\leq \sup_{p\geq 1}\frac{2\epsilon ^{1/p}}{p}=\frac{2}{%
		e\ln \left( 1/\epsilon \right) }.
	\end{equation*}%
	\bigskip 
\end{proof}

We prove parts (ii) and (iii) of Proposition \ref{Proposition vector
	concentration}.

\begin{proof}
	(ii) If $\mathcal{X}$ is a Hilbert space and the $X_{i}$ are iid, then by
	Jensen's inequality%
	\begin{equation}
	\mathbb{E}\left[ \left\Vert \sum X_{i}-\mathbb{E}\left[ X_{i}^{\prime }%
	\right] \right\Vert \right] \leq \sqrt{n\mathbb{E}\left[ \left\Vert X_{1}-%
		\mathbb{E}\left[ X_{i}^{\prime }\right] \right\Vert ^{2}\right] }=\sqrt{n}%
	\left\Vert \left\Vert X_{1}\right\Vert \right\Vert _{2}\leq 2\sqrt{n}%
	\left\Vert \left\Vert X_{1}\right\Vert \right\Vert _{\psi _{1}}.
	\label{Bound proof ii}
	\end{equation}%
	Now let $f\left( x\right) =\left\Vert \sum_{i}\left( x_{i}-\mathbb{E}\left[
	X_{1}^{\prime }\right] \right) \right\Vert $. Then as in the proof of (i)%
	\begin{equation*}
	\left\vert f_{k}\left( X\right) \left( x\right) \right\vert =\left\vert
	\left\Vert \sum_{i\neq k}x_{i}+X_{k}-n\mathbb{E}\left[ X_{1}^{\prime }\right]
	\right\Vert -\mathbb{E}\left[ \left\Vert \sum_{i\neq k}x_{i}+X_{k}^{\prime
	}-n\mathbb{E}\left[ X_{1}^{\prime }\right] \right\Vert \right] \right\vert
	\leq \mathbb{E}\left[ \left\Vert X_{k}-X_{k}^{\prime }\right\Vert |X\right] .
	\end{equation*}%
	and Lemma \ref{Lemma conditional and centering} and Theorem \ref{Theorem
		subexponential} give with probability at least $1-\delta $%
	\begin{eqnarray*}
		\left\Vert \sum_{i}X_{i}-\mathbb{E}\left[ X_{1}^{\prime }\right] \right\Vert
		&\leq &\mathbb{E}\left[ \left\Vert \sum X_{i}-\mathbb{E}\left[ X_{i}^{\prime
		}\right] \right\Vert \right] +4e\left\Vert \left\Vert X_{1}\right\Vert
		\right\Vert _{\psi _{1}}\sqrt{n\ln \left( 1/\delta \right) }+4e\left\Vert
		\left\Vert X_{1}\right\Vert \right\Vert _{\psi _{1}}\ln \left( 1/\delta
		\right)  \\
		&\leq &2\sqrt{n}\left\Vert \left\Vert X_{1}\right\Vert \right\Vert _{\psi
			_{1}}+4e\left\Vert \left\Vert X_{1}\right\Vert \right\Vert _{\psi _{1}}\sqrt{%
			n\ln \left( 1/\delta \right) }+4e\left\Vert \left\Vert X_{1}\right\Vert
		\right\Vert _{\psi _{1}}\ln \left( 1/\delta \right)  \\
		&\leq &\sqrt{n}\left\Vert \left\Vert X_{1}\right\Vert \right\Vert _{\psi
			_{1}}\left( 2+8e\sqrt{\ln \left( 1/\delta \right) }\right)  \\
		&\leq &8e\sqrt{n}\left\Vert \left\Vert X_{1}\right\Vert \right\Vert _{\psi
			_{1}}\sqrt{2\ln \left( 1/\delta \right) },
	\end{eqnarray*}%
	where the second inequality follows from (\ref{Bound proof ii}), the third
	from $n\geq \ln \left( 1/\delta \right) $, and the last from $\ln \left(
	1/\delta \right) >\ln 2$. Division by $n$ completes the proof of (ii).
	
	(iii) Apply Theorem \ref{Theorem Bernsteinoid} to $f\left( x\right)
	=\left\Vert \sum_{i}\left( x_{i}-\mathbb{E}\left[ X_{1}^{\prime }\right]
	\right) \right\Vert $ and solve for the deviation to arrive at 
	\begin{eqnarray*}
		\left\Vert \sum_{i}X_{i}-\mathbb{E}\left[ X_{1}^{\prime }\right] \right\Vert
		&\leq &\mathbb{E}\left[ \left\Vert \sum X_{i}-\mathbb{E}\left[ X_{i}^{\prime
		}\right] \right\Vert \right] +\left\Vert \left\Vert X_{1}-\mathbb{E}\left[
		X_{1}^{\prime }\right] \right\Vert \right\Vert _{2p}\sqrt{2n\ln \left(
			1/\delta \right) }+2eq\left\Vert \left\Vert X_{1}-\mathbb{E}\left[
		X_{1}^{\prime }\right] \right\Vert \right\Vert _{\psi _{1}}\ln \left(
		1/\delta \right)  \\
		&\leq &\sqrt{n}\left\Vert \left\Vert X_{1}-\mathbb{E}\left[ X_{1}^{\prime }%
		\right] \right\Vert \right\Vert _{2p}\left( 1+\sqrt{2\ln \left( 1/\delta
			\right) }\right) +4eq\left\Vert \left\Vert X_{1}\right\Vert \right\Vert
		_{\psi _{1}}\ln \left( 1/\delta \right) ,
	\end{eqnarray*}%
	where in the second inequality we bounded the last term using Lemma \ref%
	{Lemma conditional and centering} and the first term with Jensen's
	inequality as%
	\begin{equation*}
	\mathbb{E}\left[ \left\Vert \sum X_{i}-\mathbb{E}\left[ X_{i}^{\prime }%
	\right] \right\Vert \right] \leq \sqrt{n}\left\Vert \left\Vert X_{1}-\mathbb{%
		E}\left[ X_{i}^{\prime }\right] \right\Vert \right\Vert _{2}\leq \left\Vert
	\left\Vert X_{1}-\mathbb{E}\left[ X_{1}^{\prime }\right] \right\Vert
	\right\Vert _{2p}\text{,}
	\end{equation*}%
	since $p>1$. The result follows from using $\delta \leq 1/2$ and division by 
	$n$.\bigskip 
\end{proof}

We now prove the Corollary applying to linear regression.

\begin{proof}[\textbf{Proof of Corollary \protect\ref{Corollary regression}}]
	
	$\mathcal{X}$ becomes a Banach space with the norm $\left\Vert \left(
	x,z\right) \right\Vert =L\left\Vert x\right\Vert _{H}+\left\vert
	z\right\vert $. Evidently $\left\Vert \left\Vert \left( X_{1},Z_{1}\right)
	\right\Vert \right\Vert _{\psi _{1}}\leq L\left\Vert \left\Vert
	X_{1}\right\Vert \right\Vert _{\psi _{1}}+\left\Vert \left\vert
	Z_{1}\right\vert \right\Vert _{\psi _{1}}$. Then for $h\in \mathcal{H}$ 
	\begin{eqnarray*}
		h\left( x,z\right) -h\left( x^{\prime },z^{\prime }\right)  &=&\ell \left(
		\left\langle w,x\right\rangle -z\right) -\ell \left( \left\langle
		w,x^{\prime }\right\rangle -z^{\prime }\right)  \\
		&\leq &L\left\Vert x-x^{\prime }\right\Vert _{H}+\left\vert z-z^{\prime
		}\right\vert \leq \left\Vert \left( x,z\right) -\left( x^{\prime },z^{\prime
	}\right) \right\Vert ,
\end{eqnarray*}%
so $\mathcal{H}$ is uniformly Lipschitz with constant $1$. Also for an iid
sample $\left( X,Z\right) \in \mathcal{X}^{n}$ using the Lipschitz property
of $\ell $, the triangle inequality and Jensen's inequality, it is not hard
to see that 
\begin{equation*}
\mathcal{R}\left( \mathcal{H},\left( X,Z\right) \right) \leq \frac{2}{n}%
\left( L\sqrt{\sum_{i}\left\Vert X_{i}\right\Vert _{H}^{2}}+\sqrt{%
	\sum_{i}\left\vert Z_{i}\right\vert ^{2}}\right) .
\end{equation*}%
Using the iid assumption and $\left\Vert \cdot \right\Vert _{2}\leq
2\left\Vert \cdot \right\Vert _{\psi _{1}}$ we get 
\begin{equation*}
\mathbb{E}\left[ \mathcal{R}\left( \mathcal{H},\left( X,Z\right) \right) %
\right] \leq \frac{8}{\sqrt{n}}\left( L\left\Vert \left\Vert
X_{1}\right\Vert \right\Vert _{\psi _{1}}+\left\Vert \left\Vert
Z_{1}\right\Vert \right\Vert _{\psi _{1}}\right) .
\end{equation*}%
Substitution in Theorem \ref{Theorem Rademacher application} gives for $%
n\geq \ln \left( 1/\delta \right) $ with probability at least $1-\delta $%
\begin{equation*}
\sup_{h\in \mathcal{H}}\frac{1}{n}\sum_{i}h\left( X_{i},Z_{i}\right) -%
\mathbb{E}\left( h\left( X_{i},Z_{i}\right) \right) \leq \frac{8}{\sqrt{n}}%
\left( L\left\Vert \left\Vert X_{1}\right\Vert \right\Vert _{\psi
	_{1}}+\left\Vert \left\Vert Z_{1}\right\Vert \right\Vert _{\psi _{1}}\right)
\left( 1+2e\sqrt{\ln \left( 1/\delta \right) }\right) .
\end{equation*}%
\bigskip 
\end{proof}

Finally we prove the Theorem referring to metric probability spaces.

\begin{proof}[\textbf{Proof of Theorem \protect\ref{Theorem metric}}]
	The result follows easily from Theorem \ref{Theorem subexponential} and 
	\begin{align*}
	& \left\Vert f_{k}\left( X\right) \left( x\right) \right\Vert _{\psi _{1}} \\
	& =\left\Vert f\left( x_{1},\dots ,x_{k-1},X_{k},x_{k+1},\dots ,x_{n}\right)
	-\mathbb{E}\left[ f\left( x_{1},\dots ,x_{k-1},X_{k},x_{k+1},\dots
	,x_{n}\right) \right] \right\Vert _{\psi _{1}} \\
	& =\left\Vert \mathbb{E}\left[ f\left( x_{1},\dots
	,x_{k-1},X_{k},x_{k+1},\dots ,x_{n}\right) -f\left( x_{1},\dots
	,x_{k-1},X_{k}^{\prime },x_{k+1},\dots ,x_{n}\right) |X_{k}\right]
	\right\Vert _{\psi _{1}} \\
	& \leq L\left\Vert \mathbb{E}\left[ d\left( X_{k},X_{k}^{\prime }\right)
	|X_{k}\right] \right\Vert _{\psi _{1}} \\
	& \leq L\left\Vert d\left( X,X^{\prime }\right) \right\Vert _{\psi
		_{1}}=L\Delta _{1}\left( \mathcal{X},d\right) ,
	\end{align*}%
	where Lemma \ref{Lemma conditional and centering} is used in the last
	inequality.\bigskip 
\end{proof}

\end{document}